\documentclass[a4paper,10pt]{article}
\usepackage{latexsym,amssymb,amsfonts,amsmath,amsthm,amstext,color,graphicx,times}
\usepackage[mathscr]{euscript}
\usepackage{indentfirst}
\usepackage{cite}
\usepackage{mathrsfs}
\usepackage{subcaption}
\usepackage{float}
\newcommand{\R}{\mathbb{R}}

\usepackage{wrapfig}
\usepackage[all]{xy}
\usepackage{tikz}
\usepackage{hyperref}
\usepackage{animate}
\usetikzlibrary{patterns}

\theoremstyle{plain}
\newtheorem{theorem}{Theorem}[section]

\newtheorem{proposition}[theorem]{Proposition}
\newtheorem{corollary}[theorem]{Corollary}

\theoremstyle{definition}

\newtheorem{example}[theorem]{Example}

\theoremstyle{remark}
\newtheorem{remark}[theorem]{Remark}

\DeclareMathOperator{\gra}{graph}
\DeclareMathOperator{\SVIP}{SVIP}

\newcommand{\tos}{\rightrightarrows} 

\DeclareMathOperator*{\conv}{conv}
\DeclareMathOperator*{\cl}{cl}

\newcommand{\inner}[2]{\langle #1,#2 \rangle}
\title{Generalized Ordinal Nash Games: Variational Approach}

\author{ 
O. Bueno
\thanks{Universidad del Pac\'ifico.  Lima, Per\'u. Email: \texttt{\{o.buenotangoa, cotrina\_je,garcia\_yv\}@up.edu.pe}} 
\and J. Cotrina\footnotemark[2] \and Y. Garc\'ia \footnotemark[3]
}

\begin{document}

\maketitle

\begin{abstract}
	It is known that the generalized Nash equilibrium problem can be reformulated as a quasivariational inequality. Our aim in this work is to introduce a variational approach to study the existence of solutions for generalized ordinal Nash games, that is, generalized games where the player preferences are binary relations that do not necessarily admit a utility representation.
\bigskip




\end{abstract}

\section{Introduction and Motivation}
The notion of generalized Nash games was introduced by Debreu in his seminal article~\cite{Debreu}. This is an extension of the classical Nash game in which the preference and the strategy set depend on the strategy of the rival players. Some economic problems can be seen as a generalized Nash game, for instance, competitive economic equilibrium problems or exchange economy models. For more details on the history of generalized Nash games, see~\cite{Facchinei2007}.

It is known that when the preference relation is represented by a concave utility function the generalized Nash game can be reformulated as a quasivariational inequality problem, see~\cite{HARKER199181,AGM-16,ASV-2016,ACS2021,FACCHINEI2007159}. This reformulation was extended to the quasiconcave case, due to the concept of the normal cone in~\cite{Aussel-Dutta,BCC2021,ACS2021}. On the other hand, recently, Milasi \emph{et al.}~\cite{MILASI2019}, used the variational approach to guarantee the existence of solutions for a competitive economic equilibrium problem, where the consumer preferences are given through a binary relation. Similarly, Donato and Villanaci~\cite{DONATO2023} reformulated an exchange economy model as a quasivariational inequality. Thus, the purpose of this manuscript is to establish necessary and sufficient conditions to guarantee the existence of solutions for generalized ordinal Nash games using the variational approach.

The paper is organized as follows. First, in Section \ref{sec:prel}, we introduce the notations used. Then, our necessary and sufficient conditions are established in Section \ref{sec:main-section}. Moreover, we establish an existence result using a quasivariational inequality reformulation.

\def\Z{\R^m}
\section{Preliminaries}\label{sec:prel}
Formally, a Nash game consists of finite number of players, say $I$. 
Each player $\nu$ controls a variable $x^\nu$ which belongs to a strategy space $X_\nu\subset\R^{n_\nu}$. 
Moreover, for each player $\nu\in I$, its rivals are denoted by $-\nu$.
Denote by $X=\prod_{\nu\in I} X_\nu$ the Cartesian product of the strategy sets, 
%
and by $X_{-\nu}=\prod_{j\neq\nu}X_j$ the Cartesian product of the strategy sets of the rivals of player $\nu$. We assume that each player $\nu$ possesses a preference relation $\succeq_\nu$, defined on $\R^n=\prod_{\nu\in I}\R^{n_\nu}$. This preference relation, together with a rival strategy $w^{-\nu}\in X_{-\nu}$,  induces a preference relation on $\R^{n_\nu}$, defined as
\[
x^\nu\succeq_{\nu,w^{-\nu}}y^\nu\quad\iff\quad (x^\nu,w^{-\nu})\succeq_{\nu}(y^\nu,w^{-\nu}).
\]  

A strategy $\hat{x}\in X$ is a \emph{Nash equilibrium} if, for each $\nu\in I$, there is not a $x^{\nu}\in X_\nu$, such that
\[
x^{\nu}\succ_{\nu,\hat{x}^{-\nu}}\hat{x}^\nu,
\] 
where $\succ_{\nu,x^{-\nu}}$ stands for the asymmetric part of $\succeq_{\nu,x^{-\nu}}$.

In a generalized ordinal game, the strategy of each player $\nu$ must belong to a set $K_\nu(x^{-\nu})\subset X_\nu$ which depends on the rival player's strategies. A strategy $\hat{x}\in K(\hat{x})=\prod_{\nu\in I}K_\nu(x^{-\nu})$ is a \emph{ generalized Nash equilibrium} if, for each $\nu\in I$, there is not a $x^{\nu}\in K_\nu(\hat{x}^{-\nu})$ such that
\[
x^{\nu}\succ_{\nu,\hat{x}^{-\nu}}\hat{x}^\nu.
\] 

For each $\nu\in I$, we consider the sets
$U_\nu^s(x)=\{y^\nu\in \R^{n_\nu}:~y^\nu\succ_{\nu,x^{-\nu}} x^\nu\}$. It is clear that $\hat{x}\in K(\hat{x})$ is a generalized Nash equilibrium if, and only if, 
\[
U^s_\nu(\hat{x})\cap K_\nu(\hat{x}^{-\nu})=\emptyset,\mbox{ for all }\nu.
\]

\section{Variational Approach}\label{sec:main-section}

We first recall some facts on quasivariational inequalities.
Consider $X$ a subset of $\R^n$ and, $T:\R^n\tos \R^n$ and $K:X\tos X$ two set-valued maps. A vector $\hat{x}\in X$ is said to be a solution of the \emph{Stampacchia quasivariational inequality problem} if $\hat{x}\in K(\hat{x})$ and there exists $\hat{x}^*\in T(\hat{x})$, such that 
\[
\langle \hat{x}^*,y-\hat{x}\rangle\geq0,\,\forall\,y\in K(\hat{x}).
\]
The solution set of the Stampacchia quasivariational inequality problems will be denoted by $\SVIP(T,K)$.

Given  a subset $A$ of  $\R^n$ and $x\in \R^n$, the ``normal cone'' of $A$ at $x$ is the set 
\[
\mathscr{N}_A(x):=\begin{cases}
	\{x^*\in \R^n\::\:\langle x^*,y-x\rangle\leq0,~\forall y\in A\},&A\neq\emptyset,\\
	\R^n,&A=\emptyset.
\end{cases}
\]
It is usual in the literature to consider the above definition whenever $A$ convex and closed, and $x\in A$. However, we will not consider such conditions in this work.

For each $\nu\in I$, we consider  the \emph{normal cone map} $N_\nu:\R^n\tos \R^{n_\nu}$ defined as
\[
N_\nu(x):=\mathscr{N}_{U_\nu^s(x)}(x^\nu).
\]
We note that if $U^s_\nu(x)=\emptyset$ then $N_\nu(x)=\R^{n_\nu}$. 

We divided this section in two parts. In the first part, we give links between quasivariational inequalities and generalized ordinal Nash games. In the second part, we establish an existence result.
\subsection{Sufficient and Necessary Conditions}

The following result states that any solution of a certain Stampacchia quasivariational inequality problem is a generalized Nash equilibrium.
\begin{theorem}\label{T1}
	For each player $\nu\in I$, we assume that
	\begin{enumerate}
		\item $K_\nu$ is  non-empty valued;
		\item $U^s_\nu(x)$ is open, for all $x\in\R^n$.
	\end{enumerate}
	If $\hat{x}\in\SVIP(N_0,K)$, where $N_0$ is defined as $N_0(x):=\prod_{\nu\in I} N_\nu(x)\setminus\{0_\nu\}$, then $\hat x$ is a generalized Nash equilibrium.
\end{theorem}
\begin{proof}
	Assume that $\hat{x}$ is not a generalized Nash equilibrium. Then there exists a player $\nu_0$ and $x^{\nu_0}\in K_{\nu_0}(\hat{x}^{-\nu_0})$, such that
	$x^{\nu_0}\succ_{\nu_0,\hat{x}^{-\nu_0}} \hat{x}^{\nu_0}$, that is, $x^{\nu_0}\in U_{\nu_0}^s(\hat{x})$. 
	Denote $y=(x^{\nu_0},\hat x^{-\nu_0})\in K(\hat x)$. Since $\hat{x}\in\SVIP(N_0,K)$, there exists $\hat x^*\in N_0(\hat x)$ such that
	\[
	0\leq \inner{\hat{x}^*}{y-\hat{x}}= \sum_{\nu\in I}\inner{\hat{x}^{*,\nu}}{y^\nu-\hat{x}^\nu} = \inner{\hat{x}^{*,\nu_0}}{x^{\nu_0}-\hat{x}^{\nu_0}}.
	\]
	On the other hand, note that $\hat x^*\in N_0(\hat x)$ implies $0\neq \hat{x}^{*,\nu_0}\in N_{\nu_0}(\hat{x})$. Since $U^s_{\nu_0}(\hat{x})$ is open, 
	\[
	0\leq \langle \hat{x}^{*,\nu_0}, x^{\nu_0}-\hat{x}^{\nu_0}\rangle<0,
	\]
	a contradiction.
\end{proof}

We now state the converse of the previous result.
\begin{theorem}\label{T2}
	For each player $\nu\in I$, we assume that
	\begin{enumerate}
		\item $K_\nu$ is  non-empty, closed and convex valued;
		\item $U^s_\nu(x)$ is convex, for all $x\in\R^n$;
		\item $x^\nu\in \cl(U^s_\nu(x))$, for all $x\in\R^n$.
	\end{enumerate}
	If $\hat{x}$ is a generalized Nash equilibrium, then it is a solution of $\SVIP(N_0,K)$, where $N_0$ is defined in Theorem \ref{T1}.
	
\end{theorem}
\begin{proof}
	First, notice that $\hat{x}$ is a generalized Nash equilibrium if, and only if, $\hat{x}\in K(\hat{x})$ and
	$U^s_\nu(\hat{x})\cap K_\nu(\hat{x}^{-\nu})=\emptyset$, for all $\nu\in I$. Thanks to assumption {\it 3.}  the set $U^s_\nu(\hat{x})$ is non-empty. Moreover, the sets
	$U^s_\nu(\hat{x})$ and $K_\nu(\hat{x}^{-\nu})$ are convex, due to assumptions {\it 1.} and {\it 2.}.  Now, by a separation theorem, there exists $\hat{x}^{*,\nu}\in\R^{n_\nu}\setminus\{0\}$ such that
	\[
	\langle \hat{x}^{*,\nu},x^\nu-y^{\nu}\rangle\geq0,\mbox{ for all }x^\nu\in K_\nu(\hat{x}^{-\nu})\mbox{ and all }y^\nu\in U^s_\nu(\hat{x}).
	\]
	This implies $\hat{x}^{*,\nu}\in N_\nu(\hat{x})$. By assumption {\it 3.}, $\hat{x}^\nu\in \cl(U^s_\nu(\hat{x}))$ and  we obtain that
	\[
	\langle \hat{x}^{*,\nu},x^\nu-\hat{x}^{\nu}\rangle\geq0,\mbox{ for all }x^\nu\in K_\nu(\hat{x}^{-\nu}),
	\]
	which in turn implies
	\[
	\langle \hat{x}^{*},x-\hat{x}\rangle=\sum_{\nu\in I}\langle \hat{x}^{*,\nu},x^\nu-\hat{x}^{\nu}\rangle\geq0,\mbox{ for all }x\in K(\hat{x}).
	\]
\end{proof}

The following example shows that we cannot drop assumption {\it 3.} in Theorem~\ref{T2}.

\begin{example}
	For $\nu\in\{1,2\}$, we consider the relation $\succeq_\nu$ defined on $\R^2$ by
	\[
	(a,b)\succeq_\nu (x,y)\mbox{ if, and only if, }(a,b)=(x,y)=(0,0).
	\]
	Moreover, consider the strategy maps $K_1,K_2:\R\to\R$ as
	\[
	K_1(x)=K_2(x)=[-1,1].
	\]
	It is not difficult to show that
	\[
	U^s_1(x,y)=\{a\in\R: a\succ_{1,y}x\}=\emptyset\mbox{ and }U^s_2(x,y)=\{b\in\R: b\succ_{2,x}y\}=\emptyset.
	\]
	Hence, this two-player game satisfies condition~2 in Theorem~\ref{T2}, but not condition~3. On the other hand, $\hat{x}=(0,0)$ is a Nash equilibrium, but it is not a solution of the variational inequality associated to the map $N_0$ defined in Theorem~\ref{T1}.
\end{example}

\subsection{An Existence Result}
From now on, we will use the continuity definitions for set-valued maps given in~\cite[Chapter~17]{aliprantis06}. 

Motivated by the previous subsection, and similarly to \cite{Aussel-Dutta,BCC2021}, we define the set-valued map $T:\R^n\tos\R^n$ as
\begin{align}\label{mapT}
	T(x):=\prod_{\nu\in I}\conv(N_\nu(x)\cap S_\nu(0,1)).
\end{align}
We will later use the map $T$ to construct a quasivariational inequality necessary to our goals. 
We will now establish that the map $T$ is upper hemicontinuous under suitable assumptions.
\begin{proposition}\label{Prop1}
Assume that $U^s_\nu$ is lower hemicontinuous, for all $\nu\in I$. Then the map $T$,
defined in \eqref{mapT}, is upper hemicontinuous with convex and compact values. Moreover, if $U^s_\nu(x)$ 
	 is convex, for each $x$, then $T$ is non-empty valued.
\end{proposition}
\begin{proof}
	It is clear that $T(x)$ is convex and compact because is the Cartesian product of convex and compact sets. 
	Since $U^s_\nu$ is lower hemicontinuous, from Proposition~25 in~\cite{donato_variational_2023}, we obtain that $N_\nu$ is closed.
	Thus, the map $D_\nu:\R^n\tos \R^{n_\nu}$ defined by $D_\nu(x)=N_\nu(x)\cap S_\nu(0,1)$ is closed, because 
	$\gra(D_\nu)=\gra(N_\nu)\cap (\R^n\times S_\nu(0,1))$. By the Closed Graph Theorem, $D_\nu$ is upper hemicontinuous. Consequently, by Theorem 17.35 in \cite{aliprantis06}, the map $\conv(D_\nu)$ is upper hemicontinuous.
	The result now follows from Theorem 17.28 in \cite{aliprantis06}.
	
	Finally, if $U^s_\nu(x)=\emptyset$, then there is nothing to prove. Otherwise, since  $U^s_\nu(x)$
	is convex and $x^\nu\notin U^s_\nu(x)$, using a separation theorem, there exists $x^{*,\nu}\in\R^{n_\nu}\setminus\{0\}$ such that
	\[
	\langle x^{*,\nu}, y^\nu-x^\nu\rangle\leq0\mbox{ for all }y^\nu\in U^s_\nu(x).
	\]
	That means $\lambda x^{*,\nu}\in N_\nu(x)$, for all $\lambda>0$, and this implies $N_\nu(x)\cap S_\nu(0,1)\neq\emptyset$. The result follows.
\end{proof}

\begin{remark}
	In Proposition~\ref{Prop1}, to guarantee the lower hemicontinuity of $U^s_\nu$, 
	it is enough that the asymmetric relation $\succ_\nu$ is an open subset of $\R^n\times\R^n$. 
	In that sense, Bergstrom \emph{et al.}~\cite{BPR1976} gave characterizations on this topological property. On the other hand, Shafer~\cite{Shafer1974} showed that $\succ_\nu$ is an open subset of $\R^n\times\R^n$, provided that there exists a continuous function $f_\nu:\R^n\times\R^n\to\R$ satisfying $f_\nu(x,y)>0$ if, and only if, $x\succ_\nu y$. It is important to note that the openness of $\succ_\nu$ is not necessary. For instance, consider the relation $\succeq$ defined on $\R^2$ as
	\[
	(a,b)\succeq (x,y)\mbox{ if, and only if, }a\geq0\mbox{ and }b\geq y.
	\]
	It is not difficult to see that $\succ$ is not an open subset of $\R^2\times\R^2$. However,
	\[
	U^s(x,y)=\begin{cases}
		[0,+\infty[,& x<0,\\
		\emptyset,&\text{otherwise},
	\end{cases}
	\]
	which is clearly lower hemicontinuous.
\end{remark}

Finally, we are ready to state  our existence result.
\begin{theorem}\label{main-result}
	For each player $\nu\in I$, assume that
	\begin{enumerate}
		\item $X_\nu$ is non-empty, convex and compact;
		\item $K_\nu$ is continuous with  non-empty, closed and convex values;
		\item $U^s(x)$ is convex and open, for all $x\in\R^n$;
		\item $U^s_\nu$ is lower hemicontinuous.
	\end{enumerate}
	Then, there exists at least a generalized Nash equilibrium.
\end{theorem}
\begin{proof}
	By Proposition \ref{Prop1}, the map $T$ is upper hemicontinuous with convex, compact and non-empty values.
Due to Theorem~4 in~\cite{shih_generalized_1985}, the Stampacchia quasivariational inequality problem associated to $T$ and $K$ admits a solution, say $\hat{x}$. 
	If $T(\hat{x})\subset N_0(\hat{x})$, where $N_0$ is defined in Theorem \ref{T1}, the result follows from Proposition \ref{Prop1} and Theorem \ref{T1}. If there exists $\nu\in I$ such that
	$0\in \conv(N_\nu(\hat{x})\cap S_\nu(0,1))$, then there exist $x_1^{*,\nu},x_2^{*,\nu},\dots, x_m^{*,\nu} \in N_\nu(\hat{x})\cap S_\nu(0,1)$ and $t_1,t_2,\dots,t_m\in[0,1]$, such that
	\[
	0=\sum_{i=1}^m t_i x_i^{*,\nu}\mbox{ and }\sum_{i=1}^mt_i=1.
	\]
	Let $j\in\{1,2,\dots,m\}$ such that $t_j\neq0$. For such $j$ we have
	\[
	-x_j^{*,\nu}=\sum_{i\neq j} \dfrac{t_i}{t_j}x_i^{*,\nu}.
	\]
	Thus, $x_j^{*,\nu}$ and $-x_j^{*,\nu}$ belong to $N_\nu(\hat{x})$. However, since 	$U^s_\nu(\hat{x})$ is open, this implies that $x_j^{*,\nu}=0$. Thus, we get a contradiction.
\end{proof}

The following example shows that our result cannot be deduced from Theorem 3 in \cite{scalzo_existence_2022}.
\begin{example}
	Let $I=\{1,2\}$ and, for each $\nu\in I$, consider the relation $\succeq_\nu$ on $\R^2$ defined as
	\[
	x\succeq y\mbox{ if, and only if, }x^\nu\geq y^\nu.
	\]
	Clearly, $\succeq_\nu$ satisfies all assumptions of Theorem~\ref{main-result}.
	Now, for each $\nu\in I$ consider the set $X_\nu=[-1,1]$ and the maps $K_\nu:X_{-\nu}\tos X_\nu$ and $Q_\nu:\R^2\tos X_\nu$ defined respectively as $K_\nu(x^{-\nu})=X_\nu$ and
	\[
	Q_\nu(x)=\{z^\nu\in X_\nu\::\:(z^\nu,x^{-\nu})\succeq_\nu x,\mbox{ for all }z^\nu\in X_\nu\},
	\]
	with $x=(x^\nu,x^{-\nu})$. In this context, we obtain a generalized game $H=\langle X_\nu,K_\nu,Q_\nu\rangle$, according to \cite[\S 4.3]{scalzo_existence_2022}. In this case, existence of Nash equilibria is guaranteed by Theorem~\ref{main-result}. However,  as $H$ is not \emph{uniform quasi-concave}~\cite[Definition 5]{scalzo_existence_2022}, we cannot apply Theorem 3 in \cite{scalzo_existence_2022}.
\end{example}

When the preference relation $\succeq_\nu$ is defined via a utility function $\theta_\nu$, Theorem~\ref{main-result} generalizes a classical result due to Arrow and Debreu~\cite[Lemma~2.5]{Arrow-Debreu}.
\begin{corollary}
Suppose for each $\nu\in I$, $X_\nu$ is  a compact, convex and non-empty subset of  $\R^{n_\nu}$, the objective function $\theta_\nu$ is continuous and quasi-concave with respect to its variable and the constraint map $K_\nu$ is continuous with convex, compact and non-empty values. Then, there exists at least one generalized Nash equilibrium.
\end{corollary}


\end{document}